\documentclass[12pt]{article}

\usepackage{amsmath}
\usepackage{amssymb}
\usepackage{amsthm}
\usepackage{color}
\usepackage[all]{xy}
\usepackage{graphicx}

\newtheorem{thm}{Theorem}
\newtheorem{defn}{Definition}
\newtheorem{lem}{Lemma}
\newtheorem{cor}{Corollary}
\newtheorem{rem}{Remark}

\newtheorem{prop}{Proposition}
\newcommand{\C}{\mathbb C }

\def\T{\mathcal T}
\def\T2{\mathcal{AV}2}

\def\M{\mathcal M}
\def\P{\mathcal P}

\def\CC{\mathbb C}

\begin{document}

\author{Tao Chen, Yunping Jiang\footnote{Partially supported by PSC-CUNY awards and by a Simons Collaboration grant for Mathematics
and by a grant from the Institute of Mathematics and the Morningside Center of Mathematics
at Chinese Academy of Sciences.}, and Linda Keen\footnote{Partially supported by PSC-CUNY awards.}}

\date{\today}

\title{Bounded Geometry and Families of Meromorphic
Functions with Two Asymptotic Values~}

\maketitle
\begin{abstract}
In this paper we study the topological class of universal covering maps from the plane to the sphere with two removed points;  we call the elements   topological transcendental maps with two asymptotic values and we denote the space by $\T2$.
We prove that an element $f \in \T2$  with finite post-singular set
is combinatorially equivalent to a meromorphic transcendental map $g$
with constant Schwarzian derivative if and only if $f$ satisfies an analytic condition we call  {\em bounded geometry}. We plan to  relate the bounded geometry condition to topological conditions such as Levy cycles and Thurston obstructions and to the geometric condition called a canonical Thurston obstruction in a  future paper.
\end{abstract}

\section{Introduction}
\label{sec:intro}

Thurston asked the question ``when can we realize a given branched covering map as a holomorphic map in such a way that the post-critical sets correspond?" and answered it
for post-critically finite degree $d$ branched covers of the sphere~\cite{T,DH}. His theorem is that a postcritically finite degree $d\geq 2$  branched covering of
the sphere, with hyperbolic orbifold, is either combinatorially equivalent
to a rational map or there is a  topological obstruction, now called a ``Thurston obstruction''.  The rational map is unique up to conjugation by a
M\"obius transformation.

Thurston's theorem  is proved by defining an appropriate Teichm\"uller space of rational maps and a holomorphic self map of this space.  Iteration of this map converges if and only if no Thurston obstruction exists.   This method does not naturally extend to transcendental maps because the proof uses the finiteness of both the  degree and the post-critical set in a crucial way. Hubbard, Schleicher, and Shishikura~\cite{HSS} generalized Thurston's theorem to a special infinite degree family they call
``exponential type'' maps. In that paper, the authors study the  limiting behavior of quadratic differentials associated to the exponential
  functions with finite post-singular set. They use a Levy cycle condition (a special case of Thurston's topological condition) to characterize when it is possible to realize a given exponential type map with finite post-singular set as an exponential map by combinatorial equivalence. The main purpose of this paper is to use a different approach based on the framework expounded in~\cite{Ji} (see also~\cite{JZ,CJ}) to understand the characterization problem for a slight generalization of this family of infinite degree maps.

In this paper we define a class of maps called topological transcendental maps with two asymptotic values which we denote by 
$\T2$. The elements in this class are universal covering maps from the plane   to the sphere with two removed points. A meromorphic transcendental map $g$ in $\T2$ is a meromorphic function
with constant Schwarzian derivative and we denote the space of all meromorphic functions with constant Schwarzian derivative by $\M2$.
Our main result in this paper is  the characterization of an $f\in \T2$ combinatorially equivalent to a map $g\in \M2$  in terms of a condition called  {\em bounded geometry}. The main theorem is  

\medskip
\begin{thm}[Main Theorem]
\label{mainthm}
A post-singularly finite map $f$ in $\T2$ is combinatorially equivalent
to a post-singularly finite transcendental meromorphic function $g$ with constant Schwarzian derivative if and only if it has bounded geometry.
The realization is unique up to conjugation by an affine map of the plane.
\end{thm}

Our techniques involve adapting the Thurston iteration scheme to our situation.  We work with a fixed normalization.  
There are two important parts to the proof of the main theorem.  The first part is to prove that  the bounded geometry 
condition implies the iterates remain in a compact subset of the Teichm\"uller space. This analysis depends on defining 
a topological condition that constrains the iterates. The second part is to use the compactness of the iterates to prove  
that  the iteration scheme converges in the Teichm\"uller space. This part of the proof involves an analysis of quadratic differentials associated to our functions.  

The paper is organized as follows.
In \S2 we review the properties of  meromorphic functions with two asymptotic values that constitute the space $\M_2$.
In \S3, we define the family $\T2$ that consists of  topological maps modeled  on maps in $\M_2$
and show that $\M_2 \subset \T2$.
In \S4 we define {\em combinatorial equivalence} between maps in $\T2$ and
in \S5 define the Teichm\"uller space $\mathcal{T}_f$ for a map $f \in \T2$.
In \S6, we introduce the induced map $\sigma_{f}$ from the Teichm\"uller space $T_{f}$ into itself;  this is the map that defines the Thurston iteration scheme.
In \S7, we define the concept of  {\em bounded geometry}  and in \S8 we prove the necessity of the bounded geometry condition in the main theorem.
 In \S9, we give the proof of the sufficiency assuming the iterates remain in a compact subset of $T_{f}$.
In \S10.1, we define  a topological property   of the post-singularly finite map $f$  in $\T2$ in terms of the winding number of a certain closed curve. We prove that the winding number is unchanged during iteration of the map $\sigma_f$ and so provides a topological constraint on the iterates.  Finally,   in \S10.2, we show how the bounded geometry condition  together with this topological constraint implies the functions remain in a compact subset of $T_{f}$ under  the iteration to complete the proof of the main theorem.

 \medskip
{\bf Acknowledgement:} The second and the third authors are partially supported by PSC-CUNY awards. The second author is also partially supported by the collaboration grant (\#199837) from the Simons Foundation and the CUNY collaborative incentive research grant (\#1861). This research is also partially supported by the collaboration grant (\#11171121) from the NSF of China
and a collaboration grant from Academy of Mathematics and Systems Science and the Morningside
Center of Mathematics at the Chinese Academy of Sciences.

\section{The Space $\M_2$} \label{sec:merom}
In this section we define the space of meromorphic functions $\M_2$.
It is the model for the  more general space of topological functions $\T2$ that we define in the next section.
We need some standard notation and definitions:\\

\noindent ${\mathbb C}$ is the complex plane, $\hat{\mathbb C}$ is the Riemann sphere and $\C^*$ is complex plane punctured at the origin.

\begin{defn}
 Given a meromorphic function $g$, the point $v$ is a {\em logarithmic singularity for the map $g^{-1}$}
 if there is a neighborhood $U_v$ and a component $V$ of $g^{-1}(U_v \setminus \{v\}) $ such that the map
 $g: V \rightarrow U_v \setminus \{v\}$ is a holomorphic universal covering map. The point $v$ is also
 called an {\em asymptotic value } for $g$ and $V$ is called an {\em asymptotic tract for $g$}.
 A point may be an asymptotic value for more than one asymptotic tract.
 An asymptotic value may be an omitted value.
 \end{defn}

\begin{defn}  Given a meromorphic function  $g$, the point $v$ is an {\em algebraic singularity
for the map $g^{-1}$} if there is a neighborhood $U_v$ such that for every component $V_i$
of $g^{-1}(U_v ) $ the map $g: V_i \rightarrow U_v $ is a degree $d_{V_i}$ branched covering map
and $d_{V_i}>1$ for finitely many components $V_{1}, \ldots V_{n}$. For these components,
if $c_i \in V_i$ satisfies $g(c_{i})=v$ then $g'(c_i)=0$;
that is $c_i$ is a {\em critical point of $g$} for $i=1, \ldots, n$ and $v$ is a {\em critical value}.
\end{defn}

Note that by a theorem of Hurwitz, if a meromorphic function is not a homeomorphism,
it must have at least two singular points  (i.e., critical points and asymptotical values) and, by the big Picard theorem,
no transcendental meromorphic function $g:\C \rightarrow \hat\C$ can omit more than two values.

The space $\M_2$ consists of meromorphic functions whose only singular values are its omitted values.
More precisely,

\begin{defn}
The space $\M_2$ consists of meromorphic functions $g:\C \rightarrow \hat\C$
with exactly two asymptotic values and no critical values.
\end{defn}

\subsection{Examples}

Examples of functions in $\M_2$ are the exponential functions $\alpha e^{\beta z}$
and the tangent functions $\alpha \tan {i\beta z}=i\alpha \tanh{\beta z}$
where $\alpha,\beta$ are complex constants.

The asymptotic values for  the exponential functions
above are $\{0,\infty\}$;  the half plane $\Re{ \beta z} < 0$ is an asymptotic tract for $0$
and the half plane $\Re{ \beta z} >0$ is an asymptotic tract for infinity.
The asymptotic values for the tangent functions above are $\{\alpha i, -\alpha i\}$
and the asymptotic tract for $\alpha i$ is the half plane $\Im{ \beta z} >0$ while
the asymptotic tract for $-\alpha i$ is the half plane $\Im{ \beta z } < 0$.

\subsection{Nevanlinna's Theorem}

To find the form of the most general function in $\M_2$ we use a special case of a theorem of Nevanlinna~\cite{Nev}.

\begin{thm} [Nevanlinna]
Every meromorphic function $g$ with exactly $p$ asymptotic values and no critical values
has the property that its Schwarzian derivative is a polynomial of degree $p-2$.  That is
\begin{equation} \label{SCH} S(g)= \big(\frac{g''}{g'}\big)' - \frac{1}{2}\big(\frac{g''}{g'}\big)^2 = a_{p-2}z^{p-2} + \ldots a_1 z + a_0. \end{equation}
Conversely, for every polynomial $P(z)$ of degree $p-2$,  the solution to
the Schwarzian differential equation $S(g)=P(z)$ is a meromorphic function with $p$ asymptotic values and no critical values.
\end{thm}

It is easy to check that $S(\alpha e^{\beta z} )= -\frac{1}{2}\beta^2$ and $S(\alpha \tan{\frac{i\beta}{2}}{z})=-\frac{1}{2}\beta^2$.

To find all functions in $\M_2$, let $\beta \in \CC$ be constant and consider the Schwarzian differential equation
\begin{equation}\label{eqn:Schwarzian}
 S(g)=-\beta^2/2 \end{equation}
and the related  second order linear differential equation
 \begin{equation}\label{eqn:Ricci}
 w''+\frac{1}{2}S(g)w= w'' -\frac{\beta^2}{4}  w = 0. \end{equation}
It is straightforward to check that if  $w_1, w_2$ are linearly independent solutions to  equation~(\ref{eqn:Ricci}),  then $g_{\beta}
=w_2/w_1$ is a solution to equation~(\ref{eqn:Schwarzian}).

Normalizing so that $w_1(0)=1,w_1'(0)=-1, w_2(0)=1, w_2'(0)=1$ and
solving equation~(\ref{eqn:Ricci}), we have $w_1=e^{- \frac{\beta}{2} z}, w_2=e^{\frac{\beta}{2}{z}}$ as linearly independent solutions and $g_{\beta}(z)=e^{\beta z}$ as the solution to  equation~(\ref{eqn:Schwarzian}).
An arbitrary solution to equation~(\ref{eqn:Schwarzian})
 then has the form \begin{equation}\label{eqn:gensoln}
 \frac{Aw_2 +Bw_1}{Cw_2+Dw_1},  \, \, A,B,C,D \in \hat{\C}, \, AD-BC= 1  \end{equation}
and its  asymptotic values are $\{A/C,B/D\}$.

\begin{rem}
The asymptotic values are distinct and omitted.
\end{rem}

\begin{rem}
If $B=C=0, AD=1, A=\sqrt{\alpha}$ we obtain the exponential family $\{\alpha e^{\beta z}\}$  with asymptotic values at $0$ and $\infty$.  If $A=-B=\sqrt{\frac{\alpha i}{2}}, \ C=D=\sqrt{-\frac{i}{2\alpha}}$ we obtain the tangent family  $\{\alpha \tan {\frac{i\beta}{2} z}\}$ whose asymptotic values $\{ \pm \alpha i\}$ are symmetric with respect to the origin.
\end{rem}

\begin{rem}\label{rmk3}
Note that in the solutions of $S(g)= - \beta^2/2$  what appears is
$e^{\beta}$, not $\beta$;  this creates an ambiguity about which branches
of the logarithm of $e^{\beta}$ correspond to the  solution of equation~(\ref{eqn:Schwarzian}).
 In section~\ref{sec:topconst} we address this ambiguity in our situation.   We show that the topological map  we start with determines a  topological constraint which in turn, defines  the appropriate branch of the logarithm for each of the iterates in our iteration scheme.
\end{rem}

\begin{rem}
One of the basic features of the Schwarzian derivative is that it satisfies the following cocycle relation:
if $f,g$ are meromorphic functions then
$$
S(g \circ f(z)) = S(g(f)) f'(z)^2 + S(f(z)).
$$
In particular, if $T$ is a M\"obius transformation,
$S(T(z))=0$ and $S(T\circ g(z) )= S((g(z))$ so that post-composing by $T$ doesn't change the Schwarzian.
\end{rem}

In our dynamical problems  the point at infinity plays a special role and the dynamics
are invariant under post-composition by an affine map.
Thus, we may assume that all the solutions have one asymptotic
value at $0$ and that they take the value $1$ at $0$.

Since this is true for $g_{\beta}(z)=e^{\beta z}$, any solution with this normalization has the form \footnote{Notice that $g_{\alpha,\beta}$ is obtained from $g_{\beta}$ by a M\"obius transformation with determinant $1$.} 
\begin{equation}\label{eqn:gennormal}
g_{\alpha,\beta}(z)= \frac{\alpha g_{\beta}(z)}{(\alpha-\frac{1}{\alpha})g_{\beta}(z) + \frac{1}{\alpha}}
\end{equation}
where $\alpha$ is an arbitrary value in $\CC^*$. The second asymptotic value is
$\lambda=\frac{\alpha}{\alpha-\frac{1}{\alpha}}$. It takes values in $\CC \setminus \{0,1\}$.
The point at infinity is an essential singularity for all these functions.

\medskip
The parameter space $\P$ for these functions is the two complex dimensional space
$$
\P=\{\alpha, \beta \in  \CC^*  \}.
$$
The parameters define a natural complex structure for the space  $\M_2$.
The subspace of {\em entire}  functions in $\M_2$ is the one dimensional subspace
of $\P$ defined by fixing $\alpha=1$ and varying $\beta$;
$$
g_{\beta}(z)=e^{\beta z}.
$$
The tangent family has symmetric asymptotic values.
Renormalized, it forms  another one dimensional subspace of $\P$.
This is defined by fixing $\alpha = \sqrt{2}$ and varying $\beta$;
$$
g_{\frac{1}{\sqrt{2}},\beta}(z)=1+\tanh{\frac{\beta}{2} z}= \frac{\sqrt{2} e^{\beta z}}{\frac{1}{\sqrt{2}}e^{\beta z}+\frac{1}{\sqrt{2}}}.
$$
These functions have asymptotic values at $\{0,2\}$ and $g_{\frac{1}{\sqrt{2}},\beta}(0)=1.$

\begin{defn}
For $g_{\alpha,\beta}(z) \in \M_2$, the set  $\Omega=\{0,\lambda\}$ of asymptotic values is the set of singular values.
The {\em post-singular set} $P_{g}$ is defined by
$$
P_{g} = \overline{ \bigcup_{x \in \Omega} \cup_{n \geq 0} g^{n}(x) } \cup \{\infty\}.
$$
\end{defn}

Note that we include the point at infinity separately in $P_g$ because
whether or not it is an asymptotic value, it is an essential singularity
and its forward orbit is not defined.
The asymptotic values are in $P_g$ and,
since $0$ and $\lambda$ are omitted and $g_{\alpha,\beta}(0) =1 \in P_g$, $\#P_g \geq 3$.

\section{The Space $\T2$ }
We now want to consider the topological structure of functions in $\M_2$
and define $\T2$ to be the set of maps with the same topology.

\begin{defn}
Let $X$ be a simply connected open surface and let $S^2 $ be the 2-sphere.
Let $f_{a,b}:X \rightarrow S^2 \setminus \{a,b\}$ be an unbranched covering map;
that is, a universal covering map.
If $Y$ is also a simply connected open surface we say the pair $(X,f^1_{a,b})$ is equivalent
to the pair $(Y, f^2_{c,d})$ if and only if there is a homeomorphism $h:X \rightarrow Y$
such that $f^2_{c,d} \circ h= f^1_{a,b}$.
An equivalence class of such classes is called a {\em 2-asymptotic value map }
and the space of these pairs is denoted by $\T2$.
\end{defn}

Let $(X,f_{a,b})$ be a representative of a map in $\T2$.
By abuse of notation, we will  often suppress the dependence
on the equivalence class and identify $X$ with $S^2 \setminus \{\infty\}$
and refer to $f_{a,b}$ as an element of $\T2$.

By definition $f_{a,b}$ is a local homeomorphism and satisfies the following conditions:

For $v=a$ or $v=b$,  let $U_v \subset X$ be a neighborhood of $v$ whose boundary is a simple closed curve
that separates $a$ from $b$ and contains $v$ in its interior.

\begin{enumerate}
\item $f_{a,b}^{-1}(U_v \setminus \{v\})$ is connected and simply connected.
\item The restriction $f_{a,b}: f_{a,b}^{-1}(U_v \setminus \{v\}) \rightarrow (U_v \setminus \{v\})$ is a regular covering of a punctured topological disk whose degree is  infinite.
\item   $f^{-1}( \partial U_v)$ is an open curve extending to infinity in both directions.
\end{enumerate}
In analogy with meromorphic functions we say

\begin{defn}
$v$ is called a {\em logarithmic singularity} of $f_{a,b}^{-1}$ or, equivalently, an {\em asymptotic value} of $f_{a,b}$.
The domain $V_v=f_{a,b}^{-1}(U_v \setminus \{v\})$ is called an {\em asymptotic tract for $v$}.
\end{defn}

\begin{defn}
{\em  $\Omega_f=\{a,b\}$ is the set of singular values of $f_{a,b}$}.
\end{defn}

Endow $S^2$ with the standard complex structure so that it is identified with $\hat\CC$.
By the classical uniformization theorem, for any pair $(X,f_{a,b})$,
there is a  map $\pi: \CC \rightarrow  X$  such that $g_{a,b} = f_{a,b} \circ \pi $ is meromorphic.
It is called {\em the meromorphic function associated to $f_{a,b}$}.

By Nevanlinna's  theorem $S(g(z))$ is constant and moreover,

\begin{prop}\label{prop2}
If  $g(z) \in \M_2 $ with $\Omega_g=\{a,b\}$ then $g(z)=g_{a,b}(z)  \in  \T2$ and, conversely,
if $g_{a,b} \in \T2$ is meromorphic then $g_{a,b} \in \M_2$.
\end{prop}

\begin{proof}
Any $g(z) \in \M_2$ is a universal cover $g:\CC \rightarrow \hat\CC \setminus \Omega_g$
and so belongs to $\T2$.
Conversely, if $g_{a,b} \in \T2$, it is meromorphic and its only singular values are the omitted values $\{a,b\}$;
it is thus in $\M_2$.
\end{proof}

We define the post-singular set for functions in $\T2$ just as we did for functions in $\M_2$.

 \begin{defn}
 For $f=f_{a,b} \in \T2$, the {\em post-singular set} $P_{f}$ is defined by
 $$
 P_{f} = \overline{ \bigcup_{n \geq 0} f^{n}(\Omega_{f}) } \cup \{\infty\}
 $$
\end{defn}

Note that under the identification of $S^2$ with the Riemann sphere and $X$ with the complex plane,
$S^2 \setminus X$ is the point at infinity and it has no forward orbit although it may be an asymptotic value.
We therefore include it in $P_f$.

Post-composition of $f_{a,b}$ with an affine transformation $T$ results in another map in $\T2$.
In what follows, therefore, we will always assume $a=0$ and the second asymptotic value, $\lambda$,
depending on $T$ and $b$, is determined by the condition $f(0)=1$.

We will be concerned only with functions in $\T2$ such that $P_f$ is finite.
Such functions are called {\em post-singularly finite}.

\section{Combinatorial Equivalence}

In this section we define combinatorial equivalence for functions in $\T2$. Choosing representatives $(X,f_{a,b})$ of the $\T2$-equivalence classes,
we may assume $X$ is always $S^2 \setminus \{\infty \}$ and $\{0,b\}$ are the singular points for all the functions
so we will omit the subscripts denoting the omitted points in the definitions below.

\begin{defn}
Suppose $(X,f_{1}), (X,f_2) $ are representatives of  two post-singularly  finite  functions  in $\T2$, chosen as above. We say that they are {\em combinatorially equivalent}  if there are two
homeomorphisms $\phi$ and $\psi$ of   $S^2 $ onto itself fixing $\{0, \infty\}$ such that $\phi\circ f_2=f_1\circ \psi$ on $X$ and $\phi^{-1}\circ \psi$ is isotopic to the identity of $S^2$ rel $P_{f_1}$.
\end{defn}
The commutative diagram for the above definition is
\begin{equation*}
\xymatrix{X\ar[r]^\psi\ar[d]^{f_1} & X\ar[d]^{f_2}\\
S^2 \ar[r]^\phi & S^2}
\end{equation*}

\section{Teichm\"uller Space $T_{f}$.}

Let ${\mathbf M}=\{ \mu \in L^{\infty} (\hat\C)\;|\; \|\mu\|_{\infty}<1\}$ be the unit ball in the space
of all measurable functions on the Riemann sphere. Each element $\mu\in {\mathbf M}$ is called a Beltrami coefficient.
For each Beltrami coefficient $\mu$, the Beltrami equation,
$$
w_{\overline{z}}=\mu w_{z}
$$
has a unique quasiconformal  solution $w^{\mu}$ which maps $\hat{\mathbb C}$  to itself and fixes $0,1, \infty$.
Moreover, $w^{\mu}$ depends holomorphically  on $\mu$.

Let $f$ be a post-singularly finite function in $\T2$  with singular set $\Omega_f=\{0, \lambda\}$ and postsingular set $P_f$.     By definition, we have $\#(\Omega_f) = 2$ and $\#(P_f) >2$.     Since post-composition by an affine map is in the equivalence class of $f$ we may always choose a representative  such that $\{f(0)=1\} \subset P_f$;  we assume we have always made this choice.  It follows that $\lambda \neq 1$ so we always have $\{0,1, \lambda, \infty\} \subset P_f$.

The Teichm\"uller space $T( P_f)$ is defined as follows.  Given Beltrami differentials   $\mu, \nu \in {\mathbf M}$  we say that $\mu$ and $\nu$ are equivalent,  and denote this by  $\mu\sim \nu$, if $w^{\mu}$ and $w^{\nu}$ fix $0, 1, \infty$ and $(w^{\mu})^{-1}\circ w^{\nu}$ is isotopic to the identity map of $\hat\C$ rel $P_f$.   We set  $T_f=T( P_f)= {\mathbf M}/ \sim \, =\{[\mu]\}$.

There is an obvious isomorphism between  $T_{f}$ and  the classical Teichm\"uller space $Teich(R)$ of Riemann surfaces with  basepoint $R=\hat{\mathbb C}\setminus P_{f}$.  It follows that  $T_{f}$
is a finite-dimensional complex manifold so that the Teichm\"uller distance $d_{T}$ and the Kobayashi distance $d_{K}$ on $T_{f}$ coincide.  It also follows that there are always  locally quasiconformal maps in the equivalence class of $f$;  we always assume we have chosen one such as our representative.

\section{Induced Holomorphic Map $\sigma_{f}$.}
For any post-singularly finite $f$ in $\T2$, there is an induced  map $\sigma= \sigma_{f}$ from $T_{f}$ into itself given by:
$$
\sigma([\mu]) =[f^{*}\mu],
$$
where
\begin{equation}~\label{pullbackformula}
\tilde{\mu}(z) =f^{*}\mu(z) = \frac{\mu_f(z) + \mu((f(z))
\theta(z)}{1 + \overline{\mu_f (z)} \mu(f(z)) \theta(z)},  \quad \mu_f=\frac{f_{\bar z}}{f_{z}},   \quad \theta(z) =\frac{\overline{f_{z}}}{f_{z}}.
\end{equation}
Because $\sigma$ is a holomorphic map we have

\begin{lem}~\label{contractive}
For any two points $\tau$ and $\tau'$ in $T_{f}$,
$$
d_{T}\Big(\sigma(\tau), \sigma(\tau')\Big)\leq d_{T}(\tau, \tau').
$$
\end{lem}

The next lemma follows directly from the definitions.

\begin{lem}~\label{fp}
A  post-singularly finite $f$ in $\T2$ is combinatorially equivalent to a meromorphic map in $\M_2$ if and only if  $\sigma$ has a fixed point in $T_{f}$.
\end{lem}

\begin{rem}~\label{threepoints}
If $\#(P_{f})=3$,  then $T_f$ consists of a single point.  This point is trivially a fixed point for $\sigma$ so
  the main theorem holds. 
We therefore assume that $\#(P_{f})\geq 4$ in the rest of the paper.
\end{rem}

\section{Bounded Geometry.}
Let the base point of $T_f$ be the hyperbolic Riemann surface $R=\hat{\mathbb C}\setminus P_{f}$  equipped with the standard complex structure $[0]\in T_{f}$.
For  $\tau$ in $T_{f}$, denote by  $R_{\tau}$  the hyperbolic Riemann surface  $R$ equipped with the complex structure $\tau$.

A simple closed curve $\gamma\subset R$ is called non-peripheral if each component of $\hat{\mathbb C}\setminus \gamma$ contains at least two points of $P_{f}$.
Let $\gamma$ be a non-peripheral simple closed curve in $R$. For any $\tau=[\mu]\in T_{f}$, let $l_{\tau}(\gamma)$ be the hyperbolic length of the  unique closed geodesic homotopic to $\gamma$ in $R_{\tau}$.

For any $\tau_{0}\in T_{f}$, let $\tau_{n}=\sigma^{n}(\tau_{0})$, $n\geq 1$.

\begin{defn}[Hyperbolic version]
We say $f$ has bounded geometry if there is a constant $a>0$ and a point $\tau_{0}\in T_{f}$ such that
$l_{\tau_{n}}(\gamma)\geq a$ for all $n\geq 0$ and all non-peripheral simple closed curves $\gamma$ in $R$.
\end{defn}

The iteration sequence $\tau_{n}=\sigma_{f}^{n}\tau_{0}=[\mu_{n}]$
determines a sequence of subsets of $\hat\CC$
$$
P_{n} = w^{\mu_{n}}(P_{f}), \quad n=0, 1, 2, \cdots.
$$
Since all the maps $w^{\mu_{n}}$ fix $0, 1, \infty$, it follows that $0, 1, \infty\in P_{n}$.

\begin{defn}[Spherical Version]
We say $f$ has bounded geometry if there is a constant $b>0$ and a point $\tau_{0}\in T_{f}$ such that
$$
d_{sp} (p_{n},q_{n}) \geq b
$$
for all $n\geq 0$ and $p_{n}, q_{n}\in P_n$,
where
$$
d_{sp}(z,z')= \frac{|z-z'|}{\sqrt{1+|z|^{2}}\sqrt{1+|z'|^{2}}}
$$
is the spherical distance on $\hat{\mathbb C}$.
\end{defn}

Note that $d_{sp}(z, \infty) = \frac{|z|}{\sqrt{1+|z|^2}}$.  Away from infinity the spherical metric and Euclidean metrics are equivalent.  Precisely, in any bounded  $K \subset \mathbb C$, there is a constant $C>0$ which depends only on $K$ such
that
$$
C^{-1} d_{sp}(x,y) \leq |x-y|\leq Cd_{sp}(x,y)\quad  \forall  x,y \in K.
$$

The following simple lemma justifies using the term ``bounded geometry'' in both of the definitions above for $f$.

\begin{lem}~\label{sg} Consider the hyperbolic Riemann surface
$\hat{\mathbb C}\setminus X$ equipped with the standard complex structure
where $X$ is a finite subset such that $0, 1, \infty \in X$. Let $m=\# (X)\geq 3$.
Let $a>0$ be a constant. If every simple closed geodesic in $\hat{\mathbb C}
\setminus S$ has hyperbolic length greater than $a$, then there is a constant $b=b(a,m)>0$ such that the
spherical distance between any two distinct points in $S$ is bounded below by $b$.
\end{lem}

\begin{proof}
If $m=3$ there are no non-peripheral simple closed curves so in the following argument we always assume that $m\geq
4$. Let $X = \{x_{1}, \cdots, x_{m-1}\}$ and $x_{m} = \infty$ and let
$|\cdot|$ denote the Euclidean metric on ${\mathbb C}$.

Suppose $0=|x_{1}| \le \cdots \le |x_{m-1}|$. Let $M = |x_{m-1}|$.
Then $|x_{2}| \le 1$,  and we have
$$
\prod_{2 \le i \le m-2} \frac{|x_{i+1}|}{|x_{i}|} =
\frac{|x_{m-1}|}{|x_{2}|} \ge M.
$$
Hence
$$
\max_{2 \le i \le m-2} \Big\{ \frac{|x_{i+1}|}{|x_{i}|}\Big\} \ge
M^{\frac{1}{m-3}}.
$$
Let
$$
A_{i} = \{z\in {\mathbb C} \quad \Big{|}\quad |x_{i}| < z <
|x_{i+1}|\}
$$
and let $\hbox{mod}(A_{i}) = \frac{1}{2\pi} \log\frac{|x_{i+1|}}{ |{x_i}|}$ be its modulus. Then for some
integer $2\leq i_{0}\leq m_{0}-2$  if follows that
$$
\mod(A_{i_{0}}) \ge \frac{\log M}{2 \pi (m-3)}.
$$
Denote the  minimum length of closed curves  $\gamma$ in $A_{i_{0}}$, measured with respect to the hyperbolic metric on $A_{i_0}$,  by $\|\gamma\|_{A_{i_{0}}}$.
Because $A_{i_0}$ is a round annulus, the core curve realizes this minimum and we can compute its hyperbolic length as $\|\gamma\|_{A_{i_{0}}}  = \frac{\pi}{\mod(A_{i_{0}})}$.

Since  $A_{i_{0}} \subset \hat{\mathbb C}\setminus S$, the hyperbolic density on $A_{i_0}$ is smaller than the hyperbolic density on $\hat{\mathbb C}\setminus S$.  Therefore, if $l_{\tau_n}(\gamma)$ denotes the length of the shortest geodesic in the homotopy class of $\gamma$ with respect to the hyperbolic metric on $\hat{\mathbb C}\setminus S$, we have
$l_{\tau_n}(\gamma)\leq\|\gamma\|_{A_{i_{0}}} $. This implies that
$$
\frac{\pi}{l_{\tau_n}(\gamma)} \geq \mod(A_{i_{0}})\geq \frac{\log M}{2 \pi (m-3)}.
$$
Thus
$$
\log M \leq \frac{2 \pi^{2} (m-3)}{l_{\tau_n}} \leq \frac{2 \pi^{2} (m-3)}{a}.
$$
This implies that
$$
M \leq M_{0}=e^{\frac{2 \pi^{2} (m-3)}{a}}.
$$
Thus the spherical distance between $\infty$ and any finite point in
$X$ has a positive lower bound $M_{0}$ which depends only on $a$ and $m$.

Next we show that the spherical distance between any two
finite points in $X$ has a positive lower bound dependent only on
$a$ and $m$. By the equivalence of the spherical and Euclidean metrics in a bounded set in the plane,
it suffices to prove that $|x-y|$ is greater than a constant $b$ for any two finite points in $X$.

First consider $J_0 (z) =1/z$ which preserves hyperbolic length with  $0, 1,\infty\in J_0(X)$.
The above argument implies that $1/|x_{i}| \leq M_{0}$ for any $2\leq i\leq m-1$. This implies that $|x_{i}|\geq 1/M_{0}$ for any $2\leq i\leq m-1$.
Similarly, for any $x_{i}\in X$ for $2\leq i\leq m-1$, consider $J_i(z) =z/ (z-x_{i})$ which again preserves hyperbolic length.
The above argument implies that $|x_{j}/|x_{j}-x_{i}| \leq M_{0}$ for any $2\leq j\not= i\leq m-1$. This in turn
implies that $|x_{j}-x_{i}| \geq 1/M_{0}^{2}$ for any $2\leq j\not= i\leq m-1$ which proves the lemma.
\end{proof}

\section{  Necessity }
\label{sec:nec}

If $f$ is combinatorially equivalent to
$g\in \M2$, then $\sigma_{f}$ has a unique fixed point $\tau_{0}$ so that $\tau_{n}=\tau_{0}$ for all $n$.  The complex structure on $\hat{\mathbb C} \setminus P_f$  defined by $\tau_0$ induces  a hyperbolic metric on it.   The shortest geodesic in this metric gives a lower bound on the lengths  of all geodesics  so that  $f$ satisfies the hyperbolic definition of bounded geometry.

\section{  Sufficiency   assuming  compactness}
\label{sec:suff}

The proof of sufficiency of our main theorem (Theorem~\ref{mainthm}) is more complicated and needs some preparatory material.
There are two parts:  one is a compactness argument and the other is a fixed point argument.  From a conceptual point of view, the compactness of the iterates is very natural and simple. From a technical point of view, however, it is not at all obvious.  
Once one has compactness, the proof of the fixed point argument  is quite standard (see~\cite{Ji}) and works for much more general cases.  We postpone  the compactness proof to the next two sections and here give the fixed point argument.
 
Given $f\in\T2$ and given any $\tau_{0}=[\mu_{0}]\in T_f$, let $\tau_{n}=\sigma^{n} (\tau_{0}) =[\mu_{n}]$
be the sequence generated by $\sigma$. Let $w^{\mu_{n}}$ be the normalized quasiconformal map with Beltrami coefficient $\mu_{n}$.
Then
$$
g_{n} = w^{\mu_{n}}\circ f\circ (w^{\mu_{n+1}})^{-1}\in \M2
$$
since it preserves
$\mu_0$ and hence is holomorphic.
Thus iterating $\sigma$, the ``Thurston iteration'', determines a sequence $\{g_{n}\}_{n=0}^{\infty}$ of maps in $\M2$ and a sequence of subsets
$P_{f,n}=w^{\mu_{n}} (P_{f})$.
Note that $P_{f,n}$ is not, in general, the post-singular set $P_{g_{n}}$ of $g_{n}$.
If it were, we would have a fixed point of $\sigma$.

\vspace*{10pt}


Suppose $f$ is a post-singularly finite map in $\T2.$  
For any $\tau=[\mu]\in T_{f}$, $w^{\mu}$ denotes a representative normalized quasiconformal map fixing $0,1, \infty$ with Beltrami differential $\mu$; 
let $T_{\tau}$ and $T^{*}_{\tau}$ denote the respective tangent space and
cotangent space of $T_{f}$ at $\tau$.  
Then $T^{*}_{\tau}$ coincides with the space ${\mathcal Q}_{\mu}$
of integrable meromorphic quadratic differentials $q=\phi(z) dz^{2}$ on $\hat\CC$.
Integrability means that  the norm of $q$, defined as
$$
||q|| =\int_{\hat{\mathbb C}} |\phi(z)| dzd\overline{z}
$$
is finite.  The finiteness implies that   $q$ may only have poles  at points of $w^{\mu} (P_{f})$ and that these poles are simple.

Set $\tilde{\tau}=\sigma(\tau)=[\tilde{\mu}]$.   By abuse of notation, we write  $f^{-1}(P_{f})$ for  $f^{-1}(P_{f}\setminus \{\infty\})\cup\{\infty\}$.  We have the following commutative
diagram:
$$
\begin{array}{ccc} \hat{\mathbb C}\setminus f^{-1}(P_{f}) & {\buildrel w^{\tilde{{\mu}}} \over
\longrightarrow} & \hat{\mathbb C}\setminus w^{\tilde{{\mu}}} (f^{-1}(P_{f}))\cr \downarrow f
&&\downarrow g_{\mu,\tilde{\mu}}\cr \hat{\mathbb C}\setminus P_{f}& {\buildrel w^{\mu}
\over \longrightarrow} & \hat{\mathbb C}\setminus w^{\mu}(P_{f}).
\end{array}
$$
 By the definition of $\sigma$,  $\tilde{\mu} = f^*\mu$ so that the map
  $g=g_{\mu,\tilde{\mu}}=  w^{\mu} \circ f \circ  (w^{\tilde{\mu}})^{-1}$ defined on ${\mathbb C}$
is meromorphic. By remark~\ref{prop2}, $g(z)$ is in $\M_{2}$  and in particular, $g'(z) \neq 0$.

Let $\sigma_{*}=d\sigma: T_{\tau}\to T_{\tilde{\tau}}$ and $\sigma^{*}: T_{\tilde{\tau}}^* \to T_{\tau}^*$ be the respective tangent and co-tangent maps of $\sigma$.   Let $\eta$ be a tangent vector at $\tau$ so that $\tilde{\eta} = \sigma_{*} \eta$ is the corresponding tangent vector at $\tilde{\tau}$.   
 These tangent vectors can be pulled back to vectors $\xi, \tilde{\xi}$ at the origin in $T_f$ by maps 
$$
\eta = (w^{\mu})^{*}\xi \quad \hbox{and}\quad \tilde{\eta} =
(w^{\tilde{\mu}})^{*} \tilde{\xi}.
$$

This results in  the following commutative diagram,
$$
\begin{array}{ccc} \tilde{\eta}& {\buildrel (w^{\tilde{\mu}})^{*} \over
\longleftarrow} & \tilde{\xi}\cr \hbox{\hskip17pt}\uparrow f^{*} &
& \hbox{\hskip17pt}\uparrow g^{*} \cr \eta & {\buildrel (w^{\mu})^{*} \over \longleftarrow}& \xi\cr
\end{array}
$$

Now suppose $\tilde{q}$ is a co-tangent vector in
$T^{*}_{\tilde{\tau}}$ and let
$q=\sigma^{*} \tilde{q}$ be the corresponding co-tangent vector in
$T^{*}_{\tau}$. Then
$\tilde{q}=\tilde{\phi} (w) dw^{2}$ is an integrable  quadratic differential on $\hat{\mathbb C}$ whose only poles can be simple and occur at the points in $w^{\tilde{\mu}}(P_{f})$; 
$q=\phi (z) dz^{2}$ is an integrable  quadratic differential on $\hat{\mathbb C}$ whose only poles can be simple and occur at the points in 
$w^{\mu}(P_{f})$. This implies that $q=\sigma_{*}\tilde{q}$
is also the push-forward integrable quadratic differential
$$
q=g_{*}\tilde{q} =\phi (z) dz^{2}
$$
of $\tilde{q}$ by $g$.   This follows from the fact that  $w^{\tilde\mu}$ takes the tesselation of fundamental domains for $f$ to a tesselation of fundamental domains for $g$ and on each  fundamental domain $g$ is a is a homeomorphism onto $\hat{\C} \setminus \{ 0, \lambda \}$ since 
$0, \lambda $ are the two  asymptotic values of $g$.  
 The  coefficient $\phi(z)$ of $q$ is therefore given by the standard transfer operator $\mathcal L$
\begin{equation}~\label{pushforwardformula}
\phi(z)= ({\mathcal L}\tilde{\phi}) (z) =\sum_{g(w) = z}
\frac{\tilde{\phi} (w)dw^2}{(g'(w))^{2}}.
\end{equation}

Since $g'(w) \neq 0$,  equation~(\ref{pushforwardformula}) implies  the poles of  $q$ occur only at the images of the poles of $\tilde{q}$;  the integrability implies these poles can only be simple.  Therefore,   as a  meromorphic quadratic differential defined on $\hat\CC$,  $q$ satisfies 
\begin{equation}\label{qineq}
||q||\leq ||\tilde{q}||.
\end{equation}

By  formula~(\ref{pushforwardformula}) we have
$$
<\tilde{q} ,\tilde{\xi}> = <q, \xi>
$$
which, together with  inequality~(\ref{qineq}),  implies
$$
\|\tilde{\xi}\| \leq \|\xi\|.
$$
This gives another proof that $\sigma$ is weakly contracting. 
 We can, however, prove strong contraction.  

\begin{lem}~\label{infstrongcon}
$$
||q||<||\tilde{q}||
$$
and
$$
\|\tilde{\xi}\| < \|\xi\|.
$$
\end{lem}

\begin{proof}
Suppose there is a $\tilde{q}$ such that $||q||=||\tilde{q}||$
and that $Z$ is the set of poles of $\tilde{q}$.
Then, since $g$ has no critical points,  the poles of $q$ must be contained in   $g(Z)$.  
Using a change of variables on each fundamental domain we obtain the equaiities
$$
\int_{\hat{\mathbb C}} |\phi(z)|dz  \, d\overline{z} =\int_{\hat{\mathbb C}} |\tilde{\phi} (w)| \, dw d\overline{w}=\int_{\hat{\mathbb C}} \big|\frac{\tilde{\phi} (w)}{(g'(w))^{2}} \big| \, dz d\overline{z}. 
$$

The triangle inequality then implies that  at every point $z$  the argument  of $\frac{\tilde{\phi} (w)}{(g'(w))^{2}}$ is the same;  that is,  for each pair $w,w'$ with $g(w)=g(w')=z$, there is a positive real number $a_{z}$ such that 
$$
\frac{\tilde{\phi} (w)}{(g'(w))^{2}}=a_{z} \frac{\tilde{\phi} (w')}{(g'(w'))^{2}}.
$$
 Thus,  by formula~(\ref{pushforwardformula}) we see that  $||q||=||\tilde{q}||$ implies $\phi(z)= \infty$ giving us a contradiction.
\end{proof}  
\begin{rem}  What we have shown is that $||q||=||\tilde{q}||$ implies 
$$g_{*}q =\phi (g(w))  = a \tilde{q} (w)$$ and therefore
all the pre-images of all the poles are poles.  That is, 
$$
g^{-1} (g(Z)) \subset Z \cup {\Omega}_{g}.
$$
But this is a contradiction because $g^{-1}(g(Z))$ is an infinite set and $Z\cup \Omega_{g}$ is a finite set.
\end{rem}

As an immediate corollary we have strong contraction. 

\begin{cor}~\label{strongcon}
For any two points $\tau$ and $\tau'$ in $T_{f}$,
$$
d_{T}\Big(\sigma(\tau), \sigma(\tau')\Big)< d_{T}(\tau, \tau').
$$
\end{cor}

Furthermore,

\begin{prop}
If $\sigma$ has a fixed point in $T_{f}$, then this fixed point must be unique. This is equivalent to saying  that
 a post-singularly finite $f$ in $\T2$  is combinatorially equivalent to at most one map $g\in \M2$.
\end{prop}

{We can now finish the proof of the sufficiency under the assumptions 
that  $f$ has bounded geometry and that the meromorphic maps 
  defined by
\begin{equation}\label{thu iter}
g_{n}=w^{\mu_{n}}\circ f\circ (w^{\mu_{n+1}})^{-1}
\end{equation}
remain inside a compact subset of  $\M2$.

Note that if $P_f = \{0,1,\infty\}$, then $f$ is a universal covering map of $\mathbb{C}^*$ and is therefore combinatorially equivalent to $e^{2 \pi i z}$.
Thus in the following argument, we assume that $\#(P_f) \geq 4$. Then, given our normalization conventions and the bounded geometry hypothesis we see that
the functions  $g_{n}$, $n=0, 1, \ldots$ satisfy the following conditions:
\begin{itemize}
\item[1)] $m=\#(w^{\mu_{n}}(P_f))\geq 4$ is fixed.
\item[2)] $0, 1, \infty, g_{n}(1) \in w^{\mu_{n}}(P_f)$.
\item[3)] $\{ 0, 1, \infty\} \subseteq
g_{n}^{-1}(w^{\mu_{n}}(P_f))$.
\item[4)] there is a $b>0$ such that $d_{sp}(p_{n}, q_{n}) \geq b$ for any $p_{n}, q_{n}\in w^{\mu_{n}}(P_f)$.
\end{itemize}

       \vspace*{5pt}
       

%

Any integrable quadratic differential $q_{n} \in T^{*}_{\tau_{n}}$ has, at worst, simple poles in the finite set  $P_{n+1,f}= w^{\mu_{n+1}}(P_f)$.
Since $T^{*}_{\tau_{n}}$ is a finite dimensional linear space, there is a quadratic differential $q_{n, max}\in T^{*}_{\tau_{n}}$
with $\| q_{n,max}\|=1$ such that
$$
0 \leq a_{n}=\sup_{||q_{n}||=1} \|(g_{n})_{*}q_{n}|| = \|(g_{n})_{*}q_{n,max}\| <1.
$$
Moreover, by the bounded geometry condition,  the possible simple poles of $\{q_{n, max}\}_{n=1}^{\infty}$ lie in a
  compact set and hence these quadratic differentials lie in a compact subset of the  space of quadratic differentials on $\hat{\mathbb C}$ with, at worst, simples poles at $m=\#(P_{f})$ points.

Let
$$
a_{\tau_{0}} =\sup_{n\geq 0} a_{n}.
$$
Let  $\{n_{i}\}$  be a sequence of  integers such that the subsequence $a_{n_{i}}\to a_{\tau_{0}}$ as $i\to \infty$.
By our assumption of compactness, $\{g_{n_{i}}\}_{i=0}^{\infty}$ has
a convergent subsequence, (for which we  use the same notation)
that converges to a holomorphic  map $g \in \M2$.
Taking a further subsequence if necessary, we obtain a convergent sequence of sets $P_{n_{i},f} =w^{\mu_{n_{i}}}(P_{f})$  with limit set $X$.
By bounded geometry, $\#(X)=\#(P_{f})$ and  $d_{sp}(x, y) \geq b$ for any $x,y\in X$.
Thus we can find a   subsequence $\{ q_{n_{i}, max}\}$ converging to an integrable quadratic differential $q$ of norm $1$  whose only poles lie in $X$ and are simple.
Now by lemma~\ref{infstrongcon},
$$
a_{\tau_{0}} = ||g_{*} q|| <1.
$$

Thus we have proved  that there is an  $0< a_{\tau_{0}}< 1$, depending only on $b$ and $f$ and $\tau_{0}$ such that
$$
\|\sigma_{*}\| \le
\|\sigma^{*}\| \le a_{\tau_{0}}.
$$
Let $l_{0}$ be a curve connecting $\tau_{0}$ and $\tau_{1}$ in $T_{f}$ and set $l_{n}=\sigma_{f}^{n}(l_{0})$ for $n\geq 1$. Then $l=\cup_{n=0}^{\infty}l_{n}$ is a curve in $T_f$
 connecting all the points $\{\tau_{n}\}_{n=0}^{\infty}$. For each point $\tilde{\tau}_{0}\in l_{0}$, we have $a_{\tilde{\tau}_{0}} <1$. Taking the maximum  gives  a uniform $a<1$ for all points in $l_0$.  Since $\sigma$ is holomorphic, $a$ is an upper bound for all points in $l$.  Therefore,
$$
d_{T} (\tau_{n+1}, \tau_{n}) \leq a \, d_{T}(\tau_{n}, \tau_{n-1})
$$
for all $n\geq 1$.
Hence, $\{ \tau_{n}\}_{n=0}^{\infty}$ is a convergent sequence with a unique limit point $\tau_{\infty}$ in $T_{f}$ and $\tau_{\infty}$  is
a fixed point of $\sigma$.

\section{Compactness}
\label{sec:com}

The final step  in the  proof of the main theorem is to show the compactness assumption is valid.   
In the case of rational maps where the map is a branched covering of finite degree,  the bounded geometry condition guarantees 
compactness,  in the case of $f\in\T2$, however, because the map is a branched covering of infinite degree,  we need a further discussion of the topological properties of post-singular maps.   We will show that for these maps there is a topologicial constraint that, together with 
  bounded geometry condition  guarantees  compactness under the iteration process.    The point is that this constraint
 and  the bounded geometry condition together control  the size of the fundamental domains so that they are neither  too small nor too big. 

\subsection{A topological constraint.}\label{sec:topconst}
We start with $f \in \T2$;  recall $\Omega_{f}=\{0, \lambda\}$ is the set of asymptotic values of $f$ and that we have normalized so that $f(0)=1$. 
Suppose that this $f$ is post-singularly finite;  that is, $P_{f}$ is finite so that 
the orbits $\{c_{k}=f^{k} (0)\}_{k=0}^{\infty}$ and $\{c_{k}'=f^{k} (\lambda)\}_{k=0}^{\infty}$ are both finite, and thus, preperiodic.  Note that neither can be  periodic because the asymptotic values are omitted.    Consider the orbit of $0$.  Preperiodicity  means there are integers $k_{1}\geq 0$ and $l\geq 1$ such that $f^{l}(c_{k_{1}+1}) =   c_{k_{1}+1}$.  That is, 
$$
\{ c_{k_{1}+1}, \ldots, c_{k_{1}+l}\}
$$
is a periodic orbit of period $l$. Set $k_{2} =k_{1}+l$.  

Let $\gamma$ be a continuous curve connecting $c_{k_{1}}$ to $c_{k_{2}}$ in 
${\mathbb R}^{2}$ which is disjoint from $P_f$, except at its endpoints. 
Because $f(c_{k_{1}})=f(c_{k_{2}})=c_{k_{1}+1}$,  the image curve $\delta= f(\gamma)$  is a closed curve.  We can choose $\gamma$ once and for all  such that $\delta$ separates $0$ and $\lambda$;  that is, so that $\delta$ is a non-trivial curve closed curve in $\hat{\C} \setminus \{0, \lambda\}$.    The fundamental group $\pi_{1}(\hat{\C} \setminus \{0, \lambda\})={\mathbb Z}$ so    the homotopy class $\eta=[\delta]$  in the fundamental group is an integer which essentially counts the number of fundamental domains 
between $c_{k_{1}}$ and $c_{k_{2}}$  and defines a ``distance'' between the fundamental domains. The integer $\eta$ depends only on the choice of $\gamma$ and since $\gamma$ is fixed, so is $\eta$.   

We now show that $\eta$ is an invariant of the Thurston iteration procedure and is thus a topological constraint on the iterates.  
 
\medskip 
\begin{lem}\label{winding}
Given $\tau_{0}=[\mu_{0}]\in T_f$, let $\tau_{n}=\sigma^{n} (\tau_{0}) =[\mu_{n}]$
be the sequence generated by $\sigma$. Let $w^{\mu_{n}}$ be the normalized quasiconformal map with Beltrami coefficient $\mu_{n}$ Let  $\gamma_{n+1}=w^{\mu_{n+1}}(\gamma)$, $\delta_{n}= w^{\mu_{n}}(\delta)$ and $\lambda_n=w^{\mu_{n}}(\lambda)$.    Then $[\delta_{n}] \in \pi_{1}(\hat{\C} \setminus \{0, \lambda_{n}\})= \eta$ for all $n$.   
\end{lem}

\begin{proof}

The iteration defines the map 
$$
g_{n} = w^{\mu_{n}}\circ f\circ (w^{\mu_{n+1}})^{-1}\in \T2
$$
which is holomorphic since it  preserves $\mu_0$.    
The continuous curve
$$
\gamma_{n+1}=w^{\mu_{n+1}} (\gamma)
$$
goes from  $c_{k_{1}, n+1}= w^{\mu_{n+1}} (c_{k_{1}})$ to $c_{k_{2}, n+1}= w^{\mu_{n+1}} (c_{k_{2}})$.  
The image curve
$$
\delta_{n}=g_{n} (\gamma_{n+1}) =   w^{\mu_{n}}(f((w^{\mu_{n+1}})^{-1} (\gamma_{n+1})))=w^{\mu_{n}} (f (\gamma)) = w^{\mu_{n}} (\delta)
$$
is a closed curve through the point $c_{k_{1}+1, n}=w^{\mu_{n}}(c_{k_{1}+1})$.

From our normalization, it follows that 
\begin{equation}\label{g}
g_{n}(z) =g_{\alpha_{n}, \beta_{n}}(z)=\frac{\alpha_{n} e^{\beta_{n}z}}{(\alpha_{n}-\frac{1}{\alpha_{n}})e^{\beta_{n}z} +\frac{1}{\alpha_{n}}}.
\end{equation}
and $0$ is an omitted value for $g_{n}$.   
Since $\lambda_n=w^{\mu_{n}}(\lambda)$,  it is 
also omitted for $g_{n}$ and
\begin{equation}\label{omit}
\lambda_{n} =\frac{\alpha_{n}}{\alpha_{n}-\frac{1}{\alpha_{n}}}\in P_{n}.
\end{equation}
Because
$$
w^{\mu_{n}}: \hat{\C}\setminus \{0, \lambda\}\to \hat{\C}\setminus \{0, \lambda_{n}\}
$$
is a normalized homeomorphism, it preserves homotopy classes and $\eta=[\delta_{n}] \in \pi_{1}(\hat{\C} \setminus \{0, \lambda_{n}\})={\mathbb Z}$. Thus the homotopy class of $\delta_{n}$ in the space $\hat{\C} \setminus \{0, \lambda_{n}\}$ is the same throughout the iteration. 
\end{proof}
\subsection {Bounded geometry implies compactness}
\label{sec:mainthmpf}

By hypothesis $f$ has bounded geometry and by the normalization of $f$,   $\Omega_f=\{0, \lambda\}$, $f(0)=1$ so that $\{0,1, \lambda, \infty\} \subset P_f$.
Moreover the iterates $$
g_{n}=w^{\mu_{n}}\circ f\circ (w^{\mu_{n+1}})^{-1}
$$
belong to  $\M_2$.

Recall  that $P_{n} =w^{\mu_{n}} (P_{f})$ and because 
   $w^{\mu_n}$ fixes $\{0, 1, \infty\}$ for all $n\geq 0$,
$\{0, 1, \infty\}\subset P_{n}$.
By equation (\ref{g}),
$$
g_{n}(1) =w^{\mu_{n}}(f(1))=  \frac{\alpha_{n} e^{\beta_{n}}}{(\alpha_{n}-\frac{1}{\alpha_{n}})e^{\beta_{n}} +\frac{1}{\alpha_{n}}}\in P_{n}.
$$
so that
$$
\{0, 1, \lambda_{n}, g_{n}(1), \infty\} \subseteq P_{n}.
$$

If $\#(P_f)=3$, then $\lambda =\infty$ and $f(1)=1$.
In this case, $\lambda=\lambda_{n}=\infty$,  $g_{n}(1)=1$ for all $n\geq 0$  and $\#(P_{n})=3$ so that 
  $g_{n} (z)=e^{\beta_{n}z}$. 
The homotopy class of $\delta_{n}$ is always $\eta$, which is its winding number about the origin in the complex analytic sense. Thus $\beta_{n}=2 \pi i\eta$ for all $n$
and $g_{n}=e^{2\pi i \eta z}$, which is the fixed under  Thurston iteration and trivially lies in a compact  subset in $\M2$.

From now on we assume that $\#(P_{f})\geq 4$.
We first prove the compactness of the iterates in  the case that $\lambda=\infty$.  By normalization,  $\lambda_{n} =\infty$ and
$$
g_{n} (z) = e^{\beta_{n} z}
$$ 
for all $n\geq 0$.

Because $f$ has bounded geometry,  $g_{n}(1) \not=1$ has a definite spherical distance from $1$ and the sequence $\{|\beta_{n}|\}$ is bounded from below;  
that is, there is a constant $k>0$ such that
$$
k\leq |\beta_{n}|, \;\; \forall n>0. 
$$

Now we use the topological constraint to prove that the sequence $\{|\beta_{n}|\}$ is also bounded from above.
We have $g_{n}'(z)= \beta_{n} g_{n} (z)$ and the homotopy class of $\delta_{n}$ is the winding number about the origin in the complex analytic  sense, thus
$$
\eta = \frac{1}{2\pi i} \oint_{\delta_{n}} \frac{dw}{w} =\frac{1}{2\pi i} \int_{\gamma_{n}} \frac{g_{n}'(z)}{g_{n} (z)} dz
$$
$$
=\frac{\beta_{n}}{2\pi i} (c_{k_{2}, n+1}-c_{k_{1}, n+1}).
$$  
Both $c_{k_{2}, n+1}, c_{k_{1}, n+1}\in P_{n+1}$, so bounded geometry implies the constant $k>0$ above can be chosen  so that 
$$
|c_{k_{2}, n+1}-c_{k_{1},n+1}|\geq k. 
$$ 
Combining this with the formula for $\eta$ we get
$$
 |\beta_{n}| \leq \frac{\eta}{2\pi |c_{k_{2}, n+1}-c_{k_{1}, n+1}|}\leq \frac{\eta}{2\pi k}. 
$$
 and thus deduce that $\{ g_{n} (z) = e^{\beta_{n} z} \}$ forms a compact subset in $\M2$.

Now let us prove  compactness of the iterates when $\lambda\not=\infty$. 
In this case, since
$$
\lambda_{n} =\frac{\alpha_{n}}{\alpha_{n}-\frac{1}{\alpha_{n}}}\in P_{n+1}
$$
has a definite spherical distance from $0$, $1$, and $\infty$, bounded geometry implies there are  two constants $0< k<K < \infty$ such that
$$ 
k\leq |\alpha_{n}|, \;\;|\alpha_{n} -1| \leq K, \quad \forall \; n\geq 0.
$$
In this case, we have that $g_{n} (1)\not=1$ too. Since $g_{n} (1) \in P_{n+1}$, bounded geometry implies that the constant $k$ can be chosen  such that
$$
k\leq |\beta_{n}|, \;\; \forall n>0.
$$

Again we use the topological constraint to prove that $\{|\beta_{n}|\}$ is also bounded from above. 
Let 
$$
M_{n} (z) =\frac{\alpha_{n} z}{(\alpha_{n}-\frac{1}{\alpha_{n}})z +\frac{1}{\alpha_{n}}}
$$
so that $g_{n}(z) = M_{n} (e^{\beta_{n}z})$.
The map $M_{n}: \hat{\C}\setminus \{ 0, \infty\} \to \hat{\C}\setminus \{0, \lambda_{n}\}$ is a homeomorphism so it induces an isomorphism from the fundamental group $\pi_{1}(\hat{\C}\setminus \{ 0, \infty\})$ to the fundamental group  $\pi_{1}(\hat{\C}\setminus \{ 0, \lambda_{n}\})$. Thus, $\eta$ is the homotopy class $[\tilde{\delta}_{n}]$ where $\tilde{\delta}_{n}= M_{n}^{-1} (\delta_{n})$.

 Note that $\tilde{\delta}_{n}$ is the image of $\gamma_{n+1}$ under $\widetilde{g}_{n}(z)=e^{\beta_{n} z}$. Since 
$\widetilde{\delta}_{n}$ is a closed curve in $\hat{\C}\setminus \{ 0, \infty\}$, $\eta$ is the winding number of $\tilde{\delta}_{n}$ about the origin  in the complex analytic sense,  and we can compute
$$
\eta = \frac{1}{2\pi i} \oint_{\widetilde{\delta_{n}}} \frac{dw}{w} =\frac{1}{2\pi i} \int_{\gamma_{n+1}} \frac{\tilde{g}_{n}'(z)}{\widetilde{g}_{n} (z)} dz=\frac{\beta_{n}}{2\pi i} (c_{k_{2}, n+1}-c_{k_{1}, n+1}).
$$  
As above,  $c_{k_{2}, n+1}, c_{k_{1}, n+1}\in P_{n+1}$,  and by bounded geometry there is a constant $k>0$ such that 
$$
|c_{k_{2}, n+1}-c_{k_{1},n+1}|\geq k,
$$ 
 so that 
$$
 |\beta_{n}| \leq \frac{\eta}{2\pi |c_{k_{2}, n+1}-c_{k_{1}, n+1}|}\leq \frac{\eta}{2\pi k}. 
$$
This inequality proves that $\{ g_{n} (z) = g_{\alpha_{n}, \beta_{n}}(z)\}$ forms a compact subset in $\M2$.

Finally, we  have shown that in all cases  the sequence $\{g_{n}\}$ is a compact subset in $\M2$. This combined with the proof in section~\ref{sec:suff} completes the proof of sufficiency in our main theorem.

Tao Chen, Department of Mathematics, CUNY Graduate
School, New York, NY 10016. Email: chentaofdh@gmail.com

Yunping Jiang, Department of Mathematics, Queens College of CUNY,
Flushing, NY 11367 and Department of Mathematics, CUNY Graduate
School, New York, NY 10016. Email: yunping.jiang@qc.cuny.edu

Linda Keen, Department of Mathematics and Computer Science, Lehman College, CUNY, Bronx NY 10468 and
 Department of Mathematics, CUNY Graduate
School, New York, NY 10016. Email: LKeen@gc.cuny.edu

\end{document}